\documentclass[preprint,12pt]{elsarticle}

\usepackage[lmargin=2.75cm, rmargin=2.75cm,bottom=2.75cm,top=2.75cm]{geometry}
\usepackage{amsfonts}
\usepackage{amsthm}
\usepackage{amsmath}
\usepackage{amssymb}
\usepackage{verbatim}
\usepackage{psfrag}
\usepackage{latexsym,color}
 3

\newcommand{\nats}{{\mathbb N}}

\newcommand{\TM}{{\mathbb T}}

\DeclareMathOperator{\card}{Card}
\DeclareMathOperator{\alf}{Alph}
\def\A{\mathbb{A}}
\def\P{\mathcal{P}}
\def\S{\mathcal{S}}
\def\B{\mathbb{B}}

\newtheorem{thm}{Theorem}
\newtheorem{cor}[thm]{Corollary}
\newtheorem{prop}[thm]{Proposition}

\newtheorem{lemma}[thm]{Lemma}
\newtheorem{corollary}[thm]{Corollary}
\newtheorem{proposition}[thm]{Proposition}

\newtheorem{example}[thm]{Example}
\newtheorem{definition}[thm]{Definition}

\newcommand{\eps}{\varepsilon} %% EMPTY WORD

\begin{document}

\begin{frontmatter}

\title{On Generating Binary Words Palindromically}

\author[label1]{Tero Harju}
\ead{harju@utu.fi}
  
\author[label1,label2]{Mari Huova\corref{cor1}}
\ead{mahuov@utu.fi}
  
\author[label1,label3]{Luca Q. Zamboni\fnref{label4}}
 \ead{lupastis@gmail.com}
  
\fntext[label4]{Partially supported by a FiDiPro grant (137991) from the Academy of Finland and by
ANR grant {\sl SUBTILE}.}
 \cortext[cor1]{Corresponding author.}
\address[label1]{Department of Mathematics and Statistics \& FUNDIM, University of Turku, Finland}
\address[label2]{Turku Centre for Computer Science}
\address[label3]{Universit\'e de Lyon,
Universit\'e Lyon 1, CNRS UMR 5208,
Institut Camille Jordan,
43 boulevard du 11 novembre 1918,
F69622 Villeurbanne Cedex, France}

\begin{abstract}
We regard a finite word $u=u_1u_2\cdots u_n$ up
to word isomorphism as an equivalence relation on $\{1,2,\ldots , n\}$
where $i$ is equivalent to $j$ if and only if $u_i=u_j.$ Some finite words
(in particular all binary words) are generated by {\it palindromic} relations of the form
$k\sim j+i-k$ for some choice of  $1\leq i\leq j\leq n$ and $k\in \{i,i+1,\ldots ,j\}.$
That is to say, some finite words $u$ are uniquely determined up to word isomorphism
by the position and length of some of its palindromic factors.
In this paper we study the function $\mu(u)$ defined as the least number of
palindromic relations required to generate $u.$
We show that if $x$ is an infinite word such that $\mu(u)\leq 2$ for each factor $u$ of $x,$ then $x$ is ultimately periodic. On the other hand we establish the existence of non-ultimately periodic words for which $\mu(u)\leq 3$ for each factor $u$
of $x,$ and obtain a complete classification of such words on a binary alphabet (which
includes the well known class of Sturmian words). In contrast, for the Thue-Morse word, we
show that the function $\mu$ is unbounded.\end{abstract}


\end{frontmatter}

\section{Introduction}
%Suppose we are given a word $u$ of length $6$ which begins in a palindrome of length $2$
%and ends in a palindrome of length $5.$ It is readily seen that $u$ must be a coding of the word
%$aabcba.$ In other words, $u$ is necessarily word isomorphic to one of the following words
%$W=\{aabcba, aaacaa, aababa, aabbba, aaaaaa\}.$
%It is easy to see that any two distinct words in $W$ may be distinguished by one or more of their
%palindromic factors. For instance, $aaaaaa$ is distinguished as the only element of $W$
%which is a palindrome. In other words, if $x$ is a palindrome of length $6$ which begins
%in a palindrome of length $2$ and ends in a palindrome of length $5,$
%then $x$ is necessarily word isomorphic to $aaaaaa.$ Similarly,  $aababa$ is the only element
%of  $W-\{aaaaaa\}$  which ends in a palindrome of length $3,$ while
%$aaacaa$ is the only element of $W-\{aaaaaa, aababa\}$ which begins in a palindrome
%of length 3, and
%finally $aabbba$ is distinguished from $aabcba$ since $aabbba$ contains two palindromes
%of length $2$ while $aabcba$ only contains one.
%We note that $aabcba$ is the unique word in $W$ having the most number of distinct symbols.
%We shall say that $aabcba$ is generated by its palindromic prefix of length $2$ and
%its palindromic suffix of length $5.$

Palindromes play an important role in various areas of mathematics including diophantine
approximation and number theory (see for instance \cite{AdBu,Co}), discrete math (see for
instance \cite{GlJuWiZa,BrHaNiRe}), algebra (see for instance \cite{KaRe,De}),
biomathematics (see for instance \cite{KaMa}), geometric symmetry in translation surfaces
associated with various dynamical systems including interval exchange transformations (see
for instance \cite{FeZa}) and theoretical physics in the spectral properties of discrete
Schr\"odinger operators defined on quasicrystals (see for instance \cite{HoKnSi}).

Let $\A$ be a finite non-empty set of \emph{letters}. For each positive integer $n$,
let $\S(n)=\{(i,j)\, \mid \, 1\leq i\leq j\leq n\}.$   Let $u=u_1u_2\cdots u_n\in \A^n$ be a word of length $n$ with values in $\A.$ Let $\alf (u)=\{u_i\,|\, 1\leq i\leq n\}$ denote the subset of $\A$ consisting of all letters occurring in $u.$
For each pair $(i,j)\in \S(n),$ we denote by
$u[i,j]$ the {\it factor} $u_iu_{i+1}\cdots u_j.$  In case $i=j$, we write  $u[i]$ in lieu of
$u[i,i].$
Finally, let $\P$ denote the collection of all palindromes.

\begin{definition}\label{gen}
\rm Let  $u \in \A^n.$ We say $u$ is \emph{palindromically generated} if there exits $S\subseteq \S(n)$ verifying the following two conditions:
\begin{itemize}
\item $u[i,j]\in \P$ for all $(i,j)\in S$,
\item for each non-empty set $\B$ and  $v\in \B^n$, if $v[i,j]\in \P$ for all $(i,j)\in S$
then there exists a mapping  $c\colon \alf (u)\rightarrow \B$ which extends to a morphism
$c\colon \alf(u)^*\rightarrow \B^*$
of words such that $c(u)=v.$
\end{itemize}
\end{definition}

We call the elements of $S$ {\it generators} (or {\it palindromic generators}), and we say that $S$ {\it palindromically generates} $u.$ It follows immediately from Definition~\ref{gen} that if two words $u
\in \A^n$ and $v \in \B^n$  of length $n$
are palindromically generated by the same set $S\subseteq \S(n),$ then $u$ and $v$ are word isomorphic, i.e., there is a bijection
$\nu\colon \alf(u) \to \alf(v)$ which extends to a morphism of words such that $\nu(u)=v.$

\begin{example}\normalfont{%
For each letter $a\in \A,$ the word $u=a$ is palindromically generated by the empty set or by the singleton set $S=\{(1,1)\}.$  The
word $u=a^2$ is palindromically generated by any one of the following sets: $\{(1,2)\}, \{(1,2),(1,1)\}, \{(1,2),(2,2)\}, \{(1,2),(1,1),(2,2)\}.$ 
Similarly, given distinct symbols $a,b\in \A,$ the word $u=ab$ is palindromically generated by any one of the following sets: $\emptyset, \{(1,1)\}, \{(2,2)\}, \{(1,1),(2,2)\}.$ 
For $n\geq 3$, both $\{(1,n-1),(1,n)\}$ and
$\{(1,n),(2,n)\}$ palindromically generate  $a^n.$}
\end{example}

Given a word $u\in \A^n,$  let $\mu(u)$ denote the infimum of the cardinality of all sets
$S\subseteq \S(n)$ that palindromically generate $u,$ i.e.,
\[\mu(u)=\inf \{\card S\,|\,S\subseteq\S(|u|) \,\mbox{palindromically generates}\,\,u\}.\]
By convention, we put $\inf(\emptyset) = +\infty.$ 
Thus for $u\in \A^n,$ we have $\mu(u)<+\infty$ if and only if $u$ is palindromically generated, and $\mu(u)=0$ if and only if $\card\alf (u)=n.$ 
It is easily checked that $\mu(a^2)=1$  and $\mu(a^n)=2$ for $n\geq 3.$ The word $u=00101100$ is palindromically generated by  $S=\{(1,2),(2,4),(3,5),(4,7),(7,8)\}.$ One can check that  $S$ is the smallest
palindromic generating set whence  $\mu(00101100)=5.$

More generally, every binary word $u$ is palindromically generated. 
Indeed, it is readily verified that for each
nonempty word $u\in \{0,1\}^*,$ the set
\[
S_u = \{ (i,j) \mid u[i,j]=ab^ka \ \text{ for }  \{a,b\} = \{0,1\},  \ k \geq 0\}
\]
palindromically generates $u.$ 

 In contrast,  the ternary word
$abca$ is not palindromically generated by any subset $S$ of $\S(4),$ whence $\mu(abca)=+\infty.$  For this reason, we
shall primarily restrict ourselves to binary words. We note that the property of being palindromically generated is not closed under the action of a morphism: For instance, $abcd$ is palindromically generated while $abca$ is not. 

Given an infinite word $x\in \A^\nats,$ we are interested in the quantity
\[\psi (x) =\sup \{\mu(u)\,|\, u\,\,\mbox{is a factor of}\,\,x\}.\]

\noindent A first basic result states that if $x\in \A^\nats$ is aperiodic (i.e., not ultimately periodic), then $\psi (x)\geq 3$ (see Proposition \ref{aperiodic2}). This is clearly not a characterization of aperiodic words since, for instance, $\psi(x)=+\infty$ for any infinite word $x$ containing $abca$ as a factor, including the periodic word $(abc)^\omega.$ As another example, it is readily checked that $\psi((aababb)^\omega)=+\infty.$  

Let $\TM\in \{0,1\}^\nats$ denote the  {\it Thue-Morse} word
\[
\TM =0110100110010110100101100110100110010110\ldots
\]
whose origins go back to the beginning of the last century
with the works of  Axel Thue \cite{Th1}.
The $n$th entry $t_n$ of $\TM$ is defined as the sum modulo $2$ of the digits in the binary expansion of $n$. 
Alternatively, $\TM$ is the fixed point beginning in $0$
of the morphism $\tau\colon \{0,1\}^* \to \{0,1\}^*$ given by $\tau(0)=01$ and
$\tau(1)=10$. 
We prove that $\psi(\TM)=+\infty.$  Indeed, in Proposition \ref{TM} we show that for each integer $n>0$ there  exists a factor $u$ of $\TM$ with $\mu(u)\geq n$.

In this paper we establish the existence of aperiodic binary words $x$ for which
$\psi(x)=3.$  Moreover, we obtain a complete description of such words. In order to state our main results, we recall that an infinite binary word $x \in \{0,1\}^\nats$ is called {\it Sturmian} if $x$ contains exactly
$n+1$ factors of each given length $n\ge 1.$ As is well known, each Sturmian word $x$ is aperiodic and uniformly recurrent, i.e, each factor of $x$ occurs in $x$ with bounded gap.
In order to state our main
result we will make use of the following morphisms: For each subset $A \subseteq  \{0,1\},$
we denote by  $d_A\colon \{0,1\}^* \to \{0,1\}^*$ the {\it doubling morphism} defined by
the rule
\[
d_A(a)=\begin{cases}aa & \text{ if } a \in A\\
a &\text{ if } a \notin A.\end{cases}
\]

\begin{definition}\label{DS} \rm{A word $y\in \{0,1\}^\nats$ is called \emph{double Sturmian}
if $y$ is a suffix of $d_A(x)$ for some Sturmian word $x$ and  $A\subseteq \{0,1\}.$  In
particular, taking $A=\emptyset,$ it follows that every Sturmian word is double Sturmian.}
\end{definition}

\noindent We note that if $y$ is double Sturmian, then there exists a Sturmian word $x,$  a
subset $A\subseteq \{0,1\}$ and $a\in A$ such that  $d_A(x)\in \{y, ay\}.$ 

\noindent
\begin{thm}\label{main} Let  $u\in \{0,1\}^+.$ Then $\mu(u)\leq 3$ if and only if $u$ is a factor
of a double Sturmian word.
\end{thm}

\noindent Combining Theorem~\ref{main} and Proposition~\ref{aperiodic2} we deduce that:

\begin{cor}\label{infinite main} Let $x\in \{0,1\}^\nats$ be aperiodic.
Then $\psi(x)=3$ if and only if $x$ is double Sturmian. In particular, if $\psi(x)=3$ then $x$ is
uniformly recurrent. \end{cor}

\section{Preliminaries}

Throughout this section $\A$ will denote a finite non-empty set.
% and $ \# \A$ denotes the cardinality of this set.
We denote by $\A^+$ the set of all {\it words}  (or finite sequences)
of the form $w=a_1a_2\cdots a_n$ with $a_i\in \A.$ We write $|w|=n$ for the length of
$w.$ We denote the empty word by $\eps$ and set $\A^*=\A^+\cup \{\eps\}.$  Given
$w,u,v\in \A^+$ with $w=uv$, we write $v=u^{-1}w$ and $u=wv^{-1}.$  For
$w=a_1a_2\cdots a_n \in \A^n$  we denote by $w^R$ the reversal (or mirror image)
$a_n\cdots a_2a_1.$   A word $w$ is a \emph{palindrome} if $w=w^R$. Let $\P$ denote
the set of all palindromes (over any alphabet).

Given two non-empty words $u$ and $v,$ we say $u$ is a \emph{border}  of $v$ if $u$ is
both a proper prefix and a proper suffix of $v.$   If~$v$ has no borders then it is said to be
\emph{unbordered}. A word $u\in \A^+$ is called {\it Lyndon} if $u$ is lexicographically smaller than each of its cyclic conjugates relative to some linear order on $\A.$ In particular, as is well known,  each Lyndon word is both primitive and unbordered.

Let $\A^\nats$ denote the set of all infinite words with values in $\A.$ An infinite word $\omega \in \A^\nats$ is \emph{aperiodic} if it is not of the
form $\omega=uv^{\nats}=uvv\cdots$ for some words $u,v \in \A^*$ with $v \ne \eps$. The
following result was shown by Saari in \cite{Saari}. It improves an earlier result of
Ehrenfeucht and Silberger in~\cite{EhrSil}. We shall give a simplified proof of this result.

\begin{lemma}\label{Lyndon}
Each aperiodic infinite word $\omega$ contains an infinite number of Lyndon words. In
particular $\omega$ has arbitrarily long unbordered factors.
\end{lemma}
\begin{proof}
Let $\preceq$ be a lexicographic ordering of words, and suppose to the contrary that
$\omega$ contains only finitely many Lyndon factors. We  write $\omega=u_1u_2 \cdots$
where for each $i\geq 2$ we have that $u_i$ is the longest Lyndon word that is a prefix of
the suffix $(u_1 \cdots u_{i-1})^{-1}w$. Then, for all~$i$, we have $u_{i+1} \preceq u_i$
since otherwise $u_iu_{i+1}$ would be a Lyndon prefix longer than $u_i$. Thus there exists
a positive integer $j$ such that $w=u_1 \cdots u_{j-1}u_j^\omega,$ contradicting that
$\omega$ is aperiodic.

\end{proof}

Let $n$ be a positive integer, and  $\S(n)=\{(i,j)\, \mid \, 1\leq i\leq j\leq n\}.$
For each pair $I=(i,j) \in \S(n)$, we define a function \[\rho_I \colon \{i,i+1,\ldots ,j\}\rightarrow
\{i,i+1,\ldots ,j\}\] called a \emph{reflection} by $\rho_I(k)= i+j-k.$ The function $\rho_I$ is
an involution, i.e., it is a permutation that satisfies $\rho_I(\rho_I(k))=k$ for all $i\leq k\leq
j.$
%For instance, if $I=(5,11)$, then the position $7$ isreflected to $\rho_I(7)=9.$
For $u\in \A^n$ we have that $u[i,j]\in \P$ if and only if $u[k]=u[\rho_I(k)]$ for each $i\leq
k\leq j.$ If $J=(i',j')$ with $i\leq i'\leq j'\leq j,$ we denote by $\rho_I(J)$  the reflected pair
$(\rho_I(j'),\rho_I(i')).$

\noindent The following lemma gives  a reformulation of Definition~\ref{gen}:

\newcommand{\eqv}{\theta}
\begin{lemma}\label{crit}
Let $u \in \A^n$ and $S \subseteq \S(n)$. The
following conditions are equivalent:
\begin{itemize}
\item[(i)]
$S$ palindromically generates $u$.
\item[(ii)]
for each  $1\le i<j \le n$,  we have $u[i]=u[j]$ if and only if
\begin{align}
&\text{there exists a finite sequence
(or {\it path}) } (I_t)_{t=1}^r\in S^r\label{EQV}\tag{$\star$}\\
&\text{such that }   j=\rho_{I_r} \rho_{I_{r-1}} \cdots \rho_{I_1}(i).\notag
\end{align}
\end{itemize}
\end{lemma}
\begin{proof}
Define a relation $\eqv$ as follows. Let $i \eqv j$ if and only if there exists an $I \in
S$ such that $j=\rho_I(i)$. Denote by $\eqv^*$ the reflexive and transitive closure of $\eqv$.
The relation $\eqv$ is symmetric since $\rho_I$ is an involution and thus
$\eqv^*$ is an equivalence relation. Here $i \eqv^* j$ if and only if \eqref{EQV} holds for
some sequence $(I_t)_{t=1}^r$ of elements from $S$. Let then $v \in \B^n$ be any word
such that  the factor $v[i,j]$ is a palindrome for each $(i,j) \in S$. Now, $i\eqv j$ implies
$v[i]=v[j]$. Consequently, $i\eqv^*j$ implies $u[i]=u[j]$ by transitivity. It follows that the
cardinality of $\B$ is at most the number of equivalence classes of $\eqv^*$.

If $S$ palindromically generates the given word $u$, then, by the second condition of
Definition~\ref{gen},  each equivalence class of~$\eqv^*$ corresponds to a different letter
in $\A$. Therefore $u[i]=u[j]$ if and only if $i\eqv^* j$. This proves the claim from (i) to~(ii).

Suppose then that $S$ satisfies (ii). First, let $I=(i,j) \in S$, and let $i \le k \le j$. Denote
$k'=\rho_I(i)$. By (ii), we have that $u[k]=u[k']$, and therefore $u[i,j] \in \P$. Hence the
first condition of Definition~\ref{gen} holds. As for the
second condition,  if $v \in \B^n$ is such that $v[i,j]$ is a
palindrome for each $(i,j) \in S$, then the  cardinality of $\B$ is at most the cardinality of $\A$.
By (ii), we have that $v[i]=v[j]$ implies $u[i]=u[j]$ which proves the claim.
\end{proof}

\begin{definition} Suppose $S \subseteq \S(n)$ palindromically generates a word $u\in \A^n,$
and let $m\in  \{1,2,\ldots ,n\}.$ We say $m$ is a \emph{leaf of} $S$ with label $u[m]\in \A$ 
if there exists at most one pair $I =(i,j)\in S$ for which $i \leq m \leq j$ and $\rho_I(m)\neq m.$
\end{definition}

\begin{lemma}\label{heritage} Let $u\in \A^+.$ Then $\mu(v)\leq \mu(u)$
for all factors $v$ of $u.$  \end{lemma}

\begin{proof}
The result is clear if either $\mu(u)=+\infty$ or $\mu(u)=0.$  So suppose $\emptyset \neq S\subseteq \S(|u|)$ palindromically
generates $u$ and set $k=\card S.$  It suffices to show that if $u=ax=yb$ for some words $x$
and $y$ and letters $a,b \in \A,$ then $\max\{\mu(x),\mu(y)\}\leq k.$  We prove only that
$\mu(x)\leq k$ as the proof that $\mu(y)\leq k$ is completely symmetric.

Write $S=\{I_1, I_2, \ldots, I_k \}.$ If no generator $I$ in $S$ is of the form $I=(1,m),$ then clearly the set $\{(i-1,j-1)\,|\, (i,j)\in S\}\subseteq \S(|x|)$ palindromically generates $x.$  Otherwise let $m \in \nats$ be
the largest integer such that $I=(1,m) \in S$. Let $D=\{r\in \{1,2,\ldots ,k\}\,|\,
I_r=(1,q)\,\,\mbox{with}\,\,q<m\}.$ Let \[S'=S \cup \{I'_r\,|\,r\in D\}\setminus \{I_r\,|\,r\in
D\}\] where for each $r\in D$ we set $I'_r=\rho_I(I_r)=(m-q+1,m)$; see Fig. \ref{Figure 1}.

\begin{figure}[ht]
\begin{center}
\setlength{\unitlength}{0.65mm}
\begin{picture}(105,24)(10,-4)
%   \put(0,0){\framebox(120,40){}}
  \put(0,0){ \line(1,0){100}}
  \put(16,0.5){\oval(30,4)[t]}
  \put(16,5){$I_r$}
    \put(66,0.5){\oval(30,4)[t]}
  \put(62,5){$I'_r$}
 %% \put(66,0.5){\oval(30,4)[t]}
  \put(41,-0.5){\oval(80,4)[b]}
   \put(41,-8.5){$I$}
   \put(105,-1.25){}
\end{picture}
\end{center}
\caption{Reflecting the generators in $S.$}\label{Figure 1}
\end{figure}
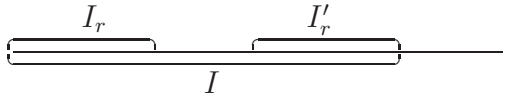

It follows that $S'$  also palindromically generates $u$ and $I$ is the only generator in $S'$
containing the initial position~$1.$ Whence $1$ is a leaf of\ $S', $ and hence  putting \[S''=S'
\cup \{(2,m-1)\}\setminus \{I\}\]  it follows that $\{(i-1,j-1)\,|\, (i,j)\in S''\}$ palindromically generates the suffix $x$ of $u.$  This
proves the claim.
\end{proof}

\noindent We now show that $\psi(x)\geq 3$ whenever $x$ is aperiodic. We will make use of the following refinement of Lemma~\ref{Lyndon}:

\begin{lemma}\label{needed} If $x\in \A^\nats$ is aperiodic, then $x$ contains arbitrarily long unbordered factors $u$ in which the first and last symbols of $u$ are not uni-occurrent in $u,$ i.e., each occurs at least twice in $u.$
\end{lemma}

\begin{proof}Choose a suffix $x'$ of $x$ such that each $a\in \alf (x')$  occurs an infinite number of times in $x'.$ Then $x'$  is also aperiodic. First assume each $a\in \alf(x')$ is {\it uniformly recurrent} in $x',$ i.e., each $a\in \alf(x')$ occurs in $x'$ with bounded gap. It follows there exists a positive integer $N$ such that each $a\in \alf(x')$ occurs at least twice in each factor $u$ of $x'$ with $|u|\geq N.$ The result now follows immediately from  Lemma~\ref{Lyndon}, which guarantees that $x'$ contains infinitely many distinct  unbordered factors.  So now suppose some symbol $a\in \alf(x')$ is not uniformly recurrent in $x'.$ Short of replacing $x'$ by a suffix of $x',$ we can assume that $a$ is the first symbol of $x'.$ Then for each positive integer $n,$ pick a factor $U_n$ of $x'$ of length $n$ which does not contain an occurrence of $a.$ Let $u(n)$ denote the shortest prefix of $x'$ ending in $U_n.$ Since $a$ does not occur in $U_n$ and $U_n$ is uni-occurrent  in $u(n),$ it follows that each $u(n)$ is unbordered. Also, since $|u(n)|\geq n,$ it follows that $\limsup _{n\rightarrow +\infty} |u(n)|=+\infty.$ Thus, there exists $N$ such that each $u(n)$ with $n\geq N$ contains at least two occurrences of $a.$ Also, by the pigeonhole principle, there exists $b\in \alf(x')$ such that infinitely many $u(n)$ terminate in $b.$  Hence there exists infinitely many $n$ with $u(n)$ of the form $u(n)=avb$ where both $a$ and $b$ occur twice in $u(n).$  
\end{proof} 

\begin{proposition}\label{aperiodic2} If $x\in \A^\nats$ is aperiodic, then $\psi (x)\geq 3.$
\end{proposition}

\begin{proof}
Let $r=\card\alf(x).$  By Lemma~\ref{needed}, $x$ contains an unbordered factor $w$ of length $|w|\geq
2r+1$ where the first and last symbols of $w$ each occurs at least twice in $w.$  We claim that $\mu(w)\geq 3.$ Suppose to the contrary that $\mu(w)\leq 2.$ If $w$ is palindromically generated by a singleton set $\{I\}\subseteq \S(|w|),$ then
$w$ contains at least $r+1$ distinct symbols, a contradiction. Otherwise, since the first and last symbols of $w$ are not uni-occurrent in $w,$ we deduce that $w$ is palindromically
generated by a set  of the form $\{(1,p),(q,|w|)\}\subseteq \S(|w|).$   As $w$ is
unbordered, it follows that the palindromic prefix $w[1,p]$ does not overlap the
palindromic suffix $w[q,|w|]$  (i.e., $p<q).$ It follows again that $w$ must have at least
$r+1$ distinct symbols, a contradiction. Thus $\mu(w)\geq 3$ and hence $\psi(x)\geq 3.$
\end{proof}

\noindent  Let $\TM\in \{0,1\}^\nats$ denote the  {\it Thue-Morse} word
\[
\TM =0110100110010110100101100110100110010110\ldots
\]

\begin{proposition}\label{TM} For each integer $n>0$ there  exists a factor $u$ of
$\TM$ with $\mu(u)\geq n,$ i.e., $\psi(\TM)=+\infty.$
\end{proposition}

\begin{proof} Let $\tau\colon \{0,1\}^* \to \{0,1\}^*$ be the morphism  fixing $\TM$ given by $\tau(0)=01$ and
$\tau(1)=10.$  We will use the well known fact that $\TM$ is {\it overlap free}, i.e.,  $\TM$  has no factors of the form $vzvzv$ where
$v$ and $z$ are nonempty words (see e.g.~\cite{Lothaire}). Set 
 $t_k=\tau^k(0)$ for $k \geq 0.$ We will show that $\mu(t_{2k})>\mu(t_{2k-2})$ for each
$k>1$ from which it follows immediately that $\psi(\TM)=+\infty.$ For each $k\geq 0$ there
exists $S_{2k}\subseteq \S(2^{2k})$ which palindromically generates $t_{2k}$  and
$\mu(t_{2k}) = \card S_{2k}$. We first observe that the prefix $t_{2k}$ of $\TM$ of length
$2^{2k}$ is a palindrome, since $t_{2k}=t_{2k-1}t^R_{2k-1}$. Since $\TM$ is
overlap-free,
it follows that if $v$ is a palindromic factor of $t_{2k}$, either $v$ lies
completely in the prefix $t_{2k-1}$ or completely in the suffix $t^R_{2k-1}$, or its midpoint
is the midpoint of $t_{2k}$.  There exists a generator in $S_{2k}$ that contains
the midpoint of $t_{2k}$ in order to identify an occurrence of $0$ in the prefix
$t_{2k-1}$ with an occurrence of $0$ in the suffix $t^R_{2k-1}$ of $t_{2k}$. Such a
generator can always be replaced by the pair $F=(1,2^{2k})$
corresponding to the full palindrome $t_{2k}$, and thus without
loss of generality we can assume that $F \in S_{2k}$.

If $I = (i,j) \in S_{2k}$ lies in the suffix  $t^R_{2k-1}$ of $t_{2k}$, i.e., if $i > 2^{2k-1},$
then we replace $I$ by its reflection $I'=\rho_F(I)$ which lies entirely in the first half of
$t_{2k}$. Since $\rho_I=\rho_F\rho_{I'}\rho_F$ on the domain of $\rho_I$, the set $(S_{2k}
\setminus \{I\}) \cup \{I'\}$ generates $t_{2k}$. In this fashion we obtain a palindromically generating set
$S'_{2k}$ consisting of $F$ and  pairs $(i,j)$ where $j \leq 2^{2k-1}$. Thus $S'_{2k}
\setminus \{F\}$ palindromically generates the prefix $t_{2k-1}$ of  $t_{2k}.$
 Since
$t_{2k-2}$ is a factor of $t_{2k-1},$ it follows from Lemma~\ref{heritage} that
$\mu(t_{2k})>\mu(t_{2k-1})\geq \mu(t_{2k-2})$ as required. \end{proof}

\noindent The following lemmas will be used in the subsequent section:

\begin{lemma}\label{double}
Suppose $u \in \A^n$ is palindromically generated by a set $S\subseteq \S(n).$ Suppose
further that there exist $p,q \in \nats$ such that $(p,p+2q) \in S.$ Then for $A=\{u[p+q]\}$
we have $\mu(d_{A}(u))\leq \card S.$ \end{lemma}

\begin{proof}
Let $a=u[p+q]$, and write $d_a$ for $d_{\{a\}}$. We define first a mapping $p_{d_a}\colon
\{1,2, \ldots, n\} \to \nats \cup \{ (j,j+1) \mid  j \in \nats \}$ for the positions of $u\in \A^n$. It
maps every position of $u$ to a corresponding position of $d_a(u)$ or to a pair of positions
of $d_a(u)$ depending on whether the position of $u$ has a label $a$ that will be doubled
or some other label. For convenience let us denote the prefix $u[1,i]$ of $u$ by $u_i$ and
let $u_0$ be the empty word.
\[
p_{d_a}(i)=\begin{cases}
|d_a(u_i)| & \text{ if } u[i] \ne a\\
(|d_a(u_i)|-1, |d_a(u_i)|) & \text{ if } u[i] = a.\end{cases}
\]
Let $S'=\{(i',j')\,|\, (i,j)\in S\}$ where $i' = |d_a(u_{i-1})|+1$ and $j'=|d_a(u_j)|$. In other
words, we dilate each generator $(i,j)$ by applying $d_a$ to the corresponding factor
$u[i,j].$ Clearly $\card S'=\card S.$ We shall show that $S'$ palindromically generates $d_a(u).$ We
claim first that if a position $i_1$ of $u$ is reflected to $i_2$ by a generator $(i_S,j_S)\in S$
then $p_{d_a}(i_1)$ is reflected to $p_{d_a}(i_2)$ by a generator $(i_{S'},j_{S'})\in S'$.
Here $(i_{S'},j_{S'})$ is the generator in $S'$ corresponding to the generator $(i_S,j_S)$ in
$S$. Now $U = u[i_S,i_1-1] = u[i_2+1,j_S]^R$, and thus $|d_a(U)| = |d_a(U^R)| =
|d_a(u[i_2+1,j_S])|$; see Figure~\ref{Figure 2}. So if $u[i_1] \ne a$ then
$p_{d_a}(i_1)-i_{S'} = j_{S'}- p_{d_a}(i_2)$ and otherwise $p_{d_a}(i_1)-(i_{S'}, i_{S'}+1) =
(j_{S'}-1,j_{S'})- p_{d_a}(i_2)$. Thus $p_{d_a}(i_1)$ is reflected to $p_{d_a}(i_2)$ as a
single or as a pair of positions.

\begin{figure}[ht]
\begin{center}
\setlength{\unitlength}{0.65mm}
\begin{picture}(200,26)(10,-12)
  \put(-5,-1.25){$u$}
    \put(0,0){ \line(1,0){80}}
  \put(10,-1){\line(0,1){2}}
    \put(8,-6){$i_S$}
    \put(10,2){\line(0,1){2}}
    \put(10,3){\line(1,0){14}}
    \put(24,2){\line(0,1){2}}
  \put(16,4){$U$}
  \put(25,0){\circle*{2}}
    \put(23,-6){$i_1$}
  \put(55,0){\circle*{2}}
    \put(53,-6){$i_2$}
    \put(56,2){\line(0,1){2}}
    \put(56,3){\line(1,0){14}}
    \put(70,2){\line(0,1){2}}
    \put(61,4){$U^R$}
    \put(70,-1){\line(0,1){2}}
    \put(68,-6){$j_S$}

    \put(90,-1.25){$d_a(u)$}
    \put(105,0){ \line(1,0){120}}
  \put(120,-1){\line(0,1){2}}
    \put(118,-7){$i_S'$}
    \put(132,3){\oval(24,3)[t]}
    \put(132,6){\oval(24,3)[b]}
  \put(122,8){$|d_a(U)|$}
  \put(145,0){\circle*{2}}
    \put(136,-7){$p_{d_a}(i_1)$}
  \put(185,0){\circle*{2}}
    \put(176,-7){$p_{d_a}(i_2)$}
    \put(198,3){\oval(24,3)[t]}
    \put(198,6){\oval(24,3)[b]}
  \put(188,8){$|d_a(U^R)|$}
    \put(210,-1){\line(0,1){2}}
    \put(208,-7){$j_S'$}
\end{picture}
\end{center}
\caption{Reflection complies with doubling morphism $d_a$.}\label{Figure 2}
\end{figure}
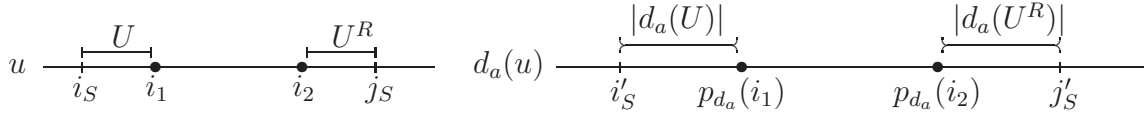

Let us denote, for any $x \in \A$,
$$\Omega_{u,x}=\{ i \mid 1\leq i \leq n \text{ and } u[i]=x \},$$
\noindent i.e., $\Omega_{u,x}$ is the set of occurrences of $x$ in $u.$ By
Lemma~\ref{crit}, for each $i,j \in \Omega_{u,x}$ there exists a sequence $i=i_1,i_2,
\ldots, i_l=j$  of positions $i_m \in \Omega_{u,x}$ such that $i_m$ is reflected to $i_{m+1}$
by some generator in~$S$. As we just showed there also exists a sequence
$i'=p_{d_a}(i_1),p_{d_a}(i_2), \ldots, p_{d_a}(i_l)=j'$ such that $p_{d_a}(i_m)$ is reflected
to $p_{d_a}(i_{m+1})$ by some generator in $S'$. Thus if $i \in \Omega_{u,x}$ then
$p_{d_a}(i) \in \hat{\Omega}_{d_a(u),x}$ where
$$\hat{\Omega}_{d_a(u),x} =\{ p_{d_a}(i) \mid i \in \Omega_{u,x}  \}.$$
\noindent In fact, if $x \ne a$ then $\hat{\Omega}_{d_a(u),x}$ is the same set as
$\Omega_{d_a(u),x}$. The only problematic set is $\hat{\Omega}_{d_a(u),a} = \{ (i', i'+1) \mid
1\leq i' < |d_a(u)| \text{ and for which } \exists i \in \Omega_{u,a} \text{ s.t. }
p_{d_a}(i)=(i', i'+1) \}$.

Now, consider the generator $(p',q') \in S'$ that is obtained from $(p,p+2q) \in S$, i.e. $(p',q')=
(|d_a(u_{p-1})|+1, |d_a(u_{p+2q})|)$. The length of the palindrome determined by the
generator $(p',q')$ is even because $|d_a(u[p,p+q-1])| = |d_a(u[p+q+1, p+2q])|$ and
$|d_a(u[p+q])|=2$. Hence a pair of two consecutive positions of $d_a(u)$, namely the
positions of $p_{d_a}(p+q) \in \hat{\Omega}_{d_a(u),a}$, is now in the middle of this
palindrome $(p',q')$ thus these two positions reflect to each other and they have to have
the same letter in word $d_a(u)$. Because of the pairwise reflections among the set
$\hat{\Omega}_{d_a(u),a}$, the positions of u according to the pairs in this set contain the
same letter $a$.

So we have that $\hat{\Omega}_{d_a(u),x} = \Omega_{d_a(u),x}$ for $x \ne a$. In the case
$x=a$ the positions covered by the pairs in $\hat{\Omega}_{d_a(u),a}$ are exactly the
positions in $\Omega_{d_a(u),a}$. This shows that there exists a path as described in
Lemma~\ref{crit} between each positions in $\Omega_{d_a(u),x}$ for any $x \in \A$ thus
$S'$ palindromically generates the doubled word $d_a(u)$ and this ends the proof.
\end{proof}

\begin{corollary}\label{dcor}
Let $a\in \A$ and  $u \in \A^n.$  Then for $A=\{a\}$ we have $\mu(d_{A}(u))\leq \mu(u)+1.$
\end{corollary}

\begin{proof}
The result is clear if the symbol $a$ does not occur in $u$ since in this case $d_{A}(u)=u.$
Thus we can assume that $a$ occurs in $u.$ Suppose $S\subseteq \S(n)$ palindromically
generates $u$ and $\mu(u)=\card S.$ If $S$ contains a generator that determines a palindrome
of odd length whose center is equal to $a,$  then  by Lemma~\ref{double} we deduce that
$\mu(d_{A}(u))\leq \mu(u)<\mu(u)+1.$   If no such  generator exists, then we can add a
"trivial" generator $(i,i)$ where $i$ is such that $u[i]=a$. Now $S \cup \{(i,i) \}$ has
$\mu(u)+1$ elements and the result follows again from Lemma~\ref{double}.
\end{proof}

\begin{example}\normalfont{
Consider a palindrome $u=aba$ which can be generated by one palindromic generator
(1,3). It is readily checked that the doubled word $d_{\{a\}}(u) = aabaa$ is generated by the set $\{(1,5),
(1,2)\}$ and $\mu(d_{\{a\}}(u)) = \mu(u)+1.$ If we instead double the letter $b$ then we do
not need any additional generators, i.e., the doubled word $d_{\{b\}}(u) = abba$ is
generated by $(1,4)$ and $\mu(d_{\{b\}}(u)) = \mu(u).$}
\end{example}

\begin{lemma}\label{leaves}
Let  $w \in \{0,1\}^n$ be an unbordered word of length $n$ in which both $0$ and $1$ are not uni-occurrent. Suppose further that $w$ is  palindromically
generated by a set $S\subseteq \S(n)$ with $\card S\leq 3.$ 
Then $S$ contains at most two leaves with label $0$ and at most two leaves with label
$1$.
\end{lemma}
\begin{proof}
Let $S\subseteq \{(1,i),(j,n),(k,m)\}$. Since $w$ is unbordered, the palindromes $w[1,i]$ and $w[j,n]$ do not overlap, and hence
we have $i < j.$ It follows each $1\leq r\leq n$ is contained in at most two generators. Suppose  $S$ contains two leaves $1\leq s<t\leq n$ with label $a$ for some $a\in \{0,1\}.$  Then there exists a shortest path $(I_q)_{q=1}^l\in S^l$ such that  $s=\rho_{I_l}\cdots \rho_{I_1}(t).$ Since each $1\leq r\leq n$ is contained in at most two generators, this path is unique and the orbit of $t$ visits every occurrence of $a$ in $u.$  Moreover, except for $s$ and $t,$  no other point $x$ in the orbit of $t$ is a leaf since as each such $x$ is contained in two generators $I$ and $J$ with $\rho_I(x)\neq x$ and $\rho_J(x)\neq x.$ 
Thus $S$ contains at most two leaves labeled $a$ for each $a\in \{0,1\}.$ 
\end{proof}

\section{A characterization of binary words with $\psi(x)=3$}

We begin by recalling a few basic facts concerning Sturmian words (see \cite{Ber, Lothaire}). An infinite binary word $\omega \in \{0,1\}^\nats$ is called {\it Sturmian} if $\omega$ contains exactly $n+1$ factors of each given length $n\ge 1.$ 
Each  Sturmian word $\omega$  is aperiodic, uniformly recurrent,  closed under reversal (i.e., if $u$ is a factor of $\omega$ then so is $u^R)$ and balanced
(for any two factors $u$ and $v$ of the same length, $||u|_a - |v|_a| \le 1$
for each $a \in \{0,1\},$ where $|w|_a$ denotes the number of occurrences of $a$ in $w)$. In fact an infinite binary word $\omega \in \{0,1\}^\nats$ is Sturmian if and only if it is aperiodic and balanced.

A factor $u$ of a Sturmian word $\omega$ is called {\it right special} (resp. {\it left special}) if both $u0$ and $u1$ (resp. $0u$ and $1u$) are factors of $\omega.$ Thus $\omega$ contains exactly one right special (resp. left special) factor of each given length, and $u$ is right special if and only if $u^R$ is left special. A factor $u$ of $\omega$ is called bispecial if $u$ is both right and left special. Thus, if $u$ is a bispecial factor of $\omega,$ then $u$ is a palindrome. 
A binary word $x\in \{0,1\}^*$ is
called a {\it central} word if and only if $x\in \P$ and $x0$ and $x1$ are both balanced.  

\noindent We recall the following fact which is a consequence of a result in \cite{deLu} (see also
\cite{CarpdeLu}):

\begin{lemma}[Proposition 22 in \cite{Ber}]\label{CenFW1} A word $x\in \{0,1\}^*$ is
a central word if it is a power of a single letter or if it satisfies the equations
$x=u01v=v10u$ with $u,v\in \{0,1\}^*.$ Moreover in the latter case $u$ and $v$ are central
words and setting $p=|u|+2$ and $q=|v|+2,$ we have that $p$ and $q$ are relatively prime
periods of $x$ and min$\{p,q\}$ is the least period of x.
\end{lemma}

\noindent We will also need the following result.

\begin{lemma}\label{unb-central}
Let $u$ be an unbordered factor of a Sturmian word $\omega\in \{0,1\}^\nats.$ Then either $u\in \{0,1\}$ or  $u=axb$, where $\{a,b\}=\{0,1\}$, and $x$ is
a central word.
\end{lemma}
\begin{proof}
If $u\notin\{0,1\},$ then we can write $u=axb$ with $\{a,b\}=\{0,1\}.$ We claim that $x$ is a palindrome.
In fact, suppose $vc$ (resp.  $dv^R)$ is a prefix (resp. suffix) of $x$ with  $v\in \{0,1\}^*$ and $\{c,d\}=\{0,1\}.$ Since $u$ is balanced it follows that $c=b$ and $d=a.$ Since $av^Rb$ and $avb$ are both factors of $\omega,$ it follows that $v$ is a bispecial factor of $\omega,$ and hence a palindrome. Thus $avb$ is both a prefix and a suffix of $u.$ Since $u$ is unbordered, it follows that $u=avb,$ and hence $x=v.$ Thus $x$ is central. 
\end{proof}

We also recall the following extremal property of the Fine and Wilf theorem~\cite{FW}
concerning central words due to de~Luca and Mignosi \cite{deLuMign}

\begin{lemma}[Proposition 23 in \cite{Ber}]\label{CenFW2} A word $x$ is a central word
if and only if there exist relatively prime positive integers $p$  and $q$  with $|x|=p+q-2$
such that $x$ has periods $p$  and~$q.$\end{lemma}

We will also use the following property by Lothaire~\cite{Lothaire-II} several times to find a
period of a palindrome which has a palindromic prefix (and suffix).

\begin{lemma}[\cite{Lothaire-II}]\label{uv}
If $uv=vu'$, then $|u|$ is a period of $uv$.
\end{lemma}

\noindent A word $y\in \{0,1\}^\nats$ is called \emph{double Sturmian}
if $y$ is a suffix of $d_A(x)$ for some Sturmian word $x$ and  $A\subseteq \{0,1\}$ (where $d_A$ is the doubling morphism defined in section 1).  By taking $A=\emptyset,$  each Sturmian word is double Sturmian.

\begin{proposition}\label{A} Let $y$ be a factor of a double Sturmian word
$\omega \in \{0,1\}^\nats.$
Then $\mu(y)\leq 3.$
\end{proposition}

\begin{proof} Let $y$ be a factor of a double Sturmian word $\omega.$
Thus there exists a Sturmian word $\omega' \in \{0,1\}^\nats$ and a subset $A\subseteq
\{0,1\}$ such that $\omega$ is a suffix of $d_A(\omega').$ Let $y'$ be the shortest
unbordered factor of $\omega'$ such that $y$ is a factor of $d_A(y')$. Because $\omega'$
is aperiodic and uniformly recurrent, such a factor $y'$ always exists. Now by
Lemma~\ref{heritage} it is enough to show that $d_A(y')$ is palindromically generated by a
set with at most three generators. 

Since $y'$ is unbordered, by Lemma~\ref{unb-central},
we may write $y'=axb$ where $x\in \{0,1\}^*$ is a central word and $\{a,b\}= \{0,1\}$.
Without loss of generality we can assume that $a=0$ and $b=1.$ We
first consider the case where $x$ is a power of a single letter. By symmetry, we can suppose
$y'=01^n$ with $n \geq 1.$ In this case $d_a(y')\in \{001^n, 01^{2n}, 001^{2n}\},$ and hence $\mu(d_A(y'))\leq 3.$ 
Next we assume $x$ is not a power of a single letter. By Lemma~\ref{double} it suffices to show that $y'$ is
palindromically generated by a set $S'$ with $\card S'\leq 3$ and containing two generators
which determine two  palindromes of odd length in $y'$ with distinct central symbols .
Now by Lemma~\ref{CenFW1} there exist
central words $u$ and $v$ such that  $y'=0u01v1.$ Put $U=0u0$ and $V=1v1.$ We claim that
$y'$ is palindromically generated by the set
\[
S'=\{(1, |U|), ( |y'|-|V|+1,|y'|), (2, |x|+1)\}.
\]
We first note that $y'[1,|U|]=U$, \ $y'[|y'|-|V|+1,|y'|]=V$ and $y'[2, |x|+1]=x, $ whence
$y'[1,|U|]$, \ $y'[|y'|-|V|+1,|y'|]$, \ $y'[2, |x|+1] \in \P.$ Next let $w\in \B^*$ be a word
with $|w|=|y'|.$ Set $U'=w[1,|U|]$, \ $V'=w[|y'|-|V|+1,|y'|],$ and $x'=w[(2, |x|+1)]$  so
that $w=a'x'b'$ with $a',b' \in \B.$ Suppose $U', V', x'\in \P.$ It follows, e.g., from
Lemma~\ref{uv} that $x'$ has now periods $|U'|=|U|$ and $|V'| = |V|$ and by
Lemma~\ref{CenFW1} they are relatively prime because $|U|$ and $|V|$ are. Since
$|x'|=|U'|+|V'|-2,$ we deduce by Lemma~\ref{CenFW2} that $x'$ is a central word. Thus
$x'$ is either a power of a single letter or it is isomorphic to $x.$ In the first case $w$ is
also a power of a single letter. In the second case, $w$ is word isomorphic to $y'.$ Thus in
either case there exists a mapping $\nu\colon \A\rightarrow \B$ with $\nu(y')=w.$ So $y'$ is
palindromically generated by $S'=\{(1, |U|), ( |y'|-|V|+1,|y'|), (2, |x|+1)\}$ where two of
the associated palindromes  have odd length. Indeed, by Lemma~\ref{CenFW1}, $|U|$ and
$|V|$ are relatively prime and $|x|=|U|+|V|-2$. So either $|U|$ and $|V|$ are odd or
only the other one is odd but then $|x|$ is odd, too. Because $y'\in \{0,1\}^*$ is
unbordered and palindromically generated by a set of size $3,$  Lemma~\ref{leaves} gives
that $y'$ has at most two leaves with label $0$ and at most two leaves with label $1$. Thus
the central symbols of these odd palindromes are necessarily distinct since  the first and the
last letter of $y'$ are also leaves with different symbols.
\end{proof}

We have an auxiliary result which shows that the generating set obtained in
Proposition~\ref{A} for $y=axb$ includes both the longest palindromic prefix and the
longest palindromic suffix of~$y$. This follows from the following refinement of
Lemma~\ref{CenFW1}.

\begin{prop}\label{LPPS} Let $x\in \{0,1\}^*$ be a central word which is not a power of
a single letter. Let $u$ and $v$ be as in Lemma~\ref{CenFW1}. Then $0u0$ (respectively
$1v1$) is the longest palindromic prefix (respectively suffix) of $0x1.$ Moreover, two of the
three palindromes $\{x, 0u0, 1v1\}$ are of odd length and have distinct central
symbols.\end{prop}

\begin{proof} Since $u$ is a central word and hence a palindrome, we have that $0u0$ is
a palindromic prefix of $0x1.$  It remains to show that it is the longest such prefix. Suppose
to the contrary that $0x1$ admits a palindromic prefix $0u'0$ with $|u'|>|u|.$ Then by
lemmas~\ref{CenFW1} and \ref{uv}, we have that $p=|u|+2, q=|v| +2$ and $p'=|x|-|u'|$
are each periods of $x.$ Since $p'=|x|-|u'|<|x|-|u|=q,$ it follows from
Lemma~\ref{CenFW1}
 that the min$\{p,q\}=p.$
Also, as $|x|=|u|+|v|+2\leq |u'| +|v| +1,$ we deduce that
\[|x|\geq |x|-|u'|+|x|-|v|-1=p'+p-1\geq p'+p-\mbox{gcd}(p,p').\]
But since both $p$ and $p'$ are periods of $x,$ it follows from the Fine and Wilf Theorem
\cite{FW} that $x$ has period $\mbox{gcd}(p,p')$ which by Lemma~\ref{CenFW1} is equal
to $p.$ Whence $p$ divides $p'.$ Let $z$ denote the suffix of $x$ of length $p'.$ Since
$0u'0$ is a palindromic prefix of $0x1$ it follows that $z$ begins in $0.$ On the other hand,
since $10u$ is a suffix of $x$ of length $p$ and $p$ divides $p',$ it follows that~$z$ begins
in~$1.$ This contradiction proves that $0u0$ is the longest palindromic prefix of~$0x1.$
Similarly one deduces that $1v1$ is the longest palindromic suffix of~$0x1.$

By Lemma~\ref{CenFW1}, $p$ and $q$ are relatively prime so two of the three palindromes
$\{x, 0u0, 1v1\}$ have odd length and by Lemma~\ref{leaves} the central symbols of these
have to be different.
\end{proof}

%\begin{example}\rm{
%Here we give an example of a word which is palindromicallyenerated
%by four palindromes can be achieved from a double
%Sturmian word. Let $y'$ be a factor of a double Sturmian word
%$d_{\{0,1\}}(\omega)=\omega' \in \{0,1\}^\nats$
%obtained from a double Sturmian word $\omega$ by doubling the letters. By
%Proposition~\ref{A},  $y'$ can be generated by three palindromes. Assume that $y'$
%contains both $0$ and $1$ and consider the word $d_{\{0\}}(y')$. Now, by Lemma~\ref{dcor},
%$d_{\{0\}}(y')$ can be generated by four palindromes and in the next section it is shown
%hat three generators are not enough unless $d_{\{0\}}(y')$ is still a factor of a double
%Sturmian word. Take for example $\varphi$ to be the Fibonacci word, i.e. $\varphi=
%010010100100101001\ldots$ and thus
%$d_{\{0,1\}}(\varphi)=001100001100110000110000110011000011\ldots$. Take e.q.\
%$y'=d_{\{0,1\}}(\varphi)[6,19] =00011001100001$. A $\P$-generating set $S$ for $y'$ is
%$S=\{ (1,12),(2,7),(9,14) \}$. Now $d_{\{0\}}(y')=00000011000011000000001$, and
%a $\P$-generating set $S'$ for $d_{\{0\}}(y')$ is $S'=\{(1,2), (1,20),(3,12),(14,23) \}$,
%where $(1,2)$ determines an added short palindrome $00.$}
%\end{example}

\noindent We next prove the converse of Theorem~\ref{main}, namely:

\begin{proposition}\label{converse}
Suppose $w\in \{0,1\}^+$ is such that $\mu(w)\leq 3.$ Then $w$ is a factor of a double Sturmian
word~$\omega.$
\end{proposition}

In what follows we use $a$ and $b$ as variables of letters such that
 $\{a,b\}=\{0,1\}$.
The key result for the converse is stated in  Lemma~\ref{zeros} which gives a simplified
criterion for pairs witnessing that a word is unbalanced when compared to the general
case of Lemma~\ref{lothaire}. We have

\begin{lemma}[Proposition~2.1.3 in \cite{Lothaire-II}]\label{lothaire}
If $w \in \{a,b\}^*$ is unbalanced then there is a palindrome $u$ such that both $aua$ and
$bub$ are factors of $w$.
\end{lemma}

\begin{lemma}\label{zeros}
Suppose $w\in \{a,b\}^*$  and $\mu(w)\leq 3.$  Suppose  $aua$ and $bub$ are 
palindromic factors of $w$. Then $u=a^ n$ or  $u=b^n$ for some even $n \geq 0.$
\end{lemma}

\begin{proof} If $u$ is empty we are done. Thus suppose $u$ is nonempty.  
Let $z$ be a factor of $w$ of minimal length containing both $aua$ and $bub.$ Then $z$ begins in one and ends in the other and both are uni-occurrent in $z.$ Since $u$ is a palindrome, the words $aua$ and $bub$ cannot overlap one another. Thus $z$ is unbordered.
By Lemma~\ref{heritage}, $z$ is palindromically generated by a subset $S\subseteq \S(|z|)$ with $\card S\leq 3. $
There exist positive integers $1<k\leq |u|+2<l<|z|$ such that $\{(1,k),(l,|z|)\}\subseteq S.$ Since $u$ is non-empty, $S$ must contain a third generator of the form $(i,j)$ for some $1<i<j<|z|.$ Thus $1$ and $|z|$ are both leaves of $S$ having distinct labels. Moreover since $(i,j)\neq (2,|z|-1)$ it follows that either $2$ or $|z|-1$ is also a leaf. Up to replacing $z$ by its reversal, we may suppose that $2$ is also a leaf. So $1,2,$ and $|z|$ are leaves. By Lemma~\ref{leaves}, $S$ contains at most one more leaf.

We claim that $|u|$ is even. In fact, suppose to the contrary that $|u|$ is odd.  If $k=|u|+2$ then $\frac{|u|+1}{2}+1$ is a leaf, while if $k<|u|+2$ then $|u|+2$ is a leaf. A similar argument applied to $l$ implies that either $|z|-(\frac{|u|+1}{2})$ or $|z|-|u|-1$ is a leaf. By Lemma~\ref{leaves} it follows that either
$2=|u|+2$ or $2=\frac{|u|+1}{2}+1.$ Since $u$ is assumed nonempty, we deduce that $2=\frac{|u|+1}{2}+1,$
i.e., $|u|=1.$ Thus either $z$ or its reverse is of the form $aaavbab$ for some $v\in \{a,b\}^*.$ We may suppose that $z=aaavbab.$ Thus $l=|z|-2$ and hence $1$ and $|z|-1$ are both leaves with label $a.$ By Lemma~\ref{leaves} we have $i=2.$ On the other hand, $j=|z|-1$ since the $a$ in position $|z|-1$ must belong to another generator other than $(|z|-2,|z|)$ since it is necessarily related to the other occurrences of $a$ in $z.$
But this implies that $z[2,|z|-1]$ is  a palindrome, a contradiction. 

Having established that $|u|$ is even, we will now show that $u$ has period equal to $1.$ We recall that we may assume without loss of generality that $1,2,$ and $|z|$ are leaves. Hence $k\in \{|u|+2, |u|+1\}$ and $l\in \{|z|-|u|-1, |z|-|u|\}.$ If either $k=|u|+1$ or $l= |z|-|u|,$ then $u$ has period $1.$ Thus we may suppose that $k=|u|+2$ and $l=|z|-|u|-1.$ Then $i=3$ or $i=4.$ In either case this  implies $j=|z|-1.$ In the first, we deduce that $u$  has period equal to $1.$ In the second case we deduce that $u$ has period $1$ or $2,$  but as $|u|$ is even, $u$ must have period equal to $1.$ This completes the proof.

\end{proof}

\begin{lemma}\label{triplets}
Suppose $w \in \{a,b\}^*$ be such that $\mu(w) \leq 3.$ Then
\begin{itemize}
\item[(i)]
The words $a^3$ and $b^3$ do not both occur in $w.$

\item[(ii)]
For odd $k$,  $ba^kb$  and $a^{k+2}$  do not both occur in $w.$

\item[(iii)]
The words $ba^kb$ and $a^{k+3}$ do not both occur in $w$ for any $k \geq 1.$

 \item[(iv)]
All three words $a^{k+2}$, $ba^{k+1}b$ and $ba^kb$ do not occur in $w$ for any $k
\geq 1.$
\end{itemize}
\end{lemma}
\begin{proof}

For Case~(i),  assume to the contrary that both $a^3$ and $b^3$ are factors of $w.$ Then $w$ contains a factor $z$ of the form 
$z=c^3ud^3$ where $\{c,d\}=\{a,b\},$  $u\in \{a,b\}^*,$ and where $c^3$ and $d^3$ are uni-occurrent in $z.$  Then $z$ is
unbordered and palindromically generated by a set $S=\{(1,k), (m,|z|),(i,j)\}$
where $k \le 3$, $m \ge |z|-2,$  $1 < i < j <|z|,$ and
either $i > 3$ or $j < |z|-2.$ This implies that $S$ either contains $3$ leaves labeled $a$ or $3$ leaves labeled $b$ contradicting
Lemma~\ref{leaves}.

Case~(ii) is merely a special case of Lemma~\ref{zeros}.

For Case~(iii), suppose to the contrary that both $ba^kb$ and $a^{k+3}$ are factors of $w.$ Then $k\geq 2$ by case $(ii).$ Then there exists $z\in \{a,b\}^+$ of the form 
$z=cd^kcud^{k+3}$ with $\{c,d\}=\{a,b\},$ $u\in \{a,b\}^*,$ with $cd^kc$ and $d^{k+3}$ uni-occurrent in $z,$ and such that $z$ or its reversal is a factor of $w.$ Thus $z$ is unbordered and palindromically generated by a set $S=\{(1,k+2), (m,|z|),(i,j)\}$
with $m \ge |z|-k-2$ and $1 < i <j < |z|.$ Thus $|z|$ is a leaf in $S$ labeled $d.$ If $i=2,$ then $|z|-1$ and $|z|-2$ are also leaves in $S$ labeled $d.$ If $i=3,$ then $2$ and $|z|-1$ are also leaves in $S$ labeled $d.$ And if $i>3,$ then
$2$ and $3$ are also leaves in $S$ labeled $d.$  In either case, we obtain a contradiction to Lemma~\ref{leaves}.

For Case~(iv), suppose to the contrary that $w$ contains $ba^kb, ba^{k+1}b,$ and $a^{k+2}$ as factors. By (ii) it follows that $k$ is even.
There are three cases to consider depending on their relative order of occurrence in $w.$ We consider only one case as the other two work analogously. In one case, $w$ contains a factor $z$ of the form $z=ba^{k+1}bua^{k+2}$  which contains an occurrence of $ba^kb$ and such that $ba^{k+1}b$ and $a^{k+2}$ are uni-occurrent in $z.$
Then $z$ is unbordered and, by Lemma~\ref{heritage}, $z$ is generated by a set $S$ with $\card S\leq 3.$ It follows that $S$ contains at least $k+2$ leaves labeled $a$ given by  $\frac{k+1}{2}+1, |z|, $ and the $k$ occurrences of $a$ arising from the occurrence in $z$ of $ba^kb.$ Since $k\geq 1, $this is a contradiction to Lemma~\ref{leaves}. 
%In the second case, $w$ contains a factor $z$ of the form $z=ba^{k+1}buba^{k}b$  which contains an occurrence of $a^{k+2}$ and such that $ba^{k+1}b$ and $ba^{k}b$ are uni-occurrent in $z.$
%Then $z$ is unbordered and, by Lemma~\ref{heritage}, $z$ is generated by a set $S$ with $\card S\leq 3.$ It follows that $S$ contains at least $k+3$ leaves labeled $a$ given by  $\frac{k+1}{2}+1 $ and the $k+2$ occurrences of $a$ arising from the occurrence in $z$ of $a^{k+2}.$ This is a contradiction to Lemma~\ref{leaves}. Finally in the third case, $w$ contains a factor $z$ of the form $z=a^{k+2}uba^{k}b$  which contains an occurrence of $a^{k+1}$ and such that $a^{k+2}$ and $ba^{k}b$ are uni-occurrent in $z.$
%Then $z$ is unbordered and, by Lemma~\ref{heritage}, $z$ is generated by a set $S$ with $\card S\leq 3.$ It is easily checked that this case too  contradicts Lemma~\ref{leaves}.
\end{proof}

%\goodbreak

For $w\in \{0,1\}^+\cup \{0,1\}^\nats$, we define $A(w) \subseteq \{0,1\}$ by the rule $ a
\in A(w)$ if and only if
 $w$ has no factors of the form $ba^{2k+1}b$ for any $k\ge 0$ with $a \ne b$.
If $w$ is  a finite word,  we define the \emph{lean word} of $w$ to be the shortest word
$u$ such that $d_{A(w)}(u)$ contains $w$ as a factor. We extend this notion to infinite
aperiodic words $w\in \{0,1\}^\nats:$  we say $u\in \{0,1\}^\nats$ is the \emph{lean word} if $d_{A(w)}(u)$
 contains $w$ as a suffix and for all $v\in \{0,1\}^\nats,$ if $w$ is a suffix of $d_{A(w)}(v),$
 then $u$ is a suffix of $v.$ In each case the lean word of $w$ is uniquely determined by $w$.

For instance, if $w=0010011$, then $A(w)=\{0\}$ and the lean word of $w$ is $u=01011.$
Similarly if $w=011001$, then $A(w)=\{0,1\}$ and the lean word $u=0101.$ For
$w=0010110$, we have $A(w)= \emptyset$ and hence $w$ is its own lean word.

\begin{lemma}\label{conlemma}
Let $w\in \{0,1\}^*$ be a binary word with $\mu(w) \le 3$. Then the lean word $u$ of $w$
is balanced.
\end{lemma}

\begin{proof}
Let $w$ be a shortest counter example to the claim such that its lean word $u$ is
unbalanced. By appealing to symmetry and Lemma~\ref{triplets}(i), we can assume
that $w$ contains no occurrences of~$111.$

Since $u$ is not balanced, it follows from Lemma~\ref{lothaire} that $u$ contains factors $0v0$ and $1v1$ for some palindrome~$v.$
Now, also $0d_{A(w)}(v)0$ and $1d_{A(w)}(v)1$ are
palindromes, and they are factors of~$w$. By Lemma~\ref{zeros}, we have
$d_{A(w)}(v)=0^k$ for some even integer~$k$. It follows that $v=0^m$ for $m=k/2$ or
$m=k$ depending on whether $0$ is in $A(w)$ or not. Here $u$ has factors $0^{m+2}$ and
$10^m1$.

Suppose first that $k > 0$. Since $w$ has factors $0^{k+2}$ and $10^k1$, it has no blocks
of $0$'s of odd length by Lemma~\ref{triplets}(iii) and~(iv). Hence $0 \in A(w)$
by the definition of $A(w)$,  but then $d_{A(w)}(u)$ has factors $0^{2m+4}$ and
$10^{2m}1$, and thus $w$ has factors $0^{2m+3}$ and $10^{2m}1$ contradicting
Lemma~\ref{triplets}(iii).

Suppose then that $k=0$. Since now $11$ occurs in $u$ but $111$ does not occur in $w$,
we have $1\notin A(w)$. By the definition of $A(w)$, the word $010$ must be a factor of
$u$, and hence by the minimality of $w$, we have $u=00(10)^r11$ for some positive $r
\geq 1$. From this also $0 \notin A(w)$ follows, and hence  $u=w$. However, the
palindromic generators of $w$ are now $(1,2)$ that determines the prefix $w[1,2]=00$,
and $(n-1,n)$ that determines the suffix $w[n-1,n]=11$, and $(i,j)$ that determines the
factor $w[i,j]$. But the palindrome $w[i,j]$ cannot overlap with both $w[1,2]$
and~$w[n-1,n]$ and thus the two~$0$'s in $w[1,2]$ and the latter~$0$ are not equivalent or
the two~$1$'s in $w[n-1,n]$ are not equivalent to a preceding $1$; a contradiction.
\end{proof}

\begin{proof} [Proof of Proposition~\ref{converse}]
By Lemma~\ref{conlemma}, if $\mu(w) \leq 3$,  then $w$ is a factor of the word
$d_{A(w)}(u)$ for a balanced word~$u.$ Since the factors of Sturmian words are exactly
the balanced words, see~\cite{Lothaire-II}, the claim follows. This also proves
Theorem~\ref{main}.
\end{proof}

\begin{figure}
\begin{center}
\setlength{\unitlength}{1mm}
\begin{picture}(80,40)(5,-7)
%   \put(0,0){\framebox(120,40){}}
  \put(0,20){\line(1,0){80}}
  \put(84,19){$d_{A(x)}(y)$}
  \put(20,1){\line(1,0){40}}
  \put(64,0){$y$}
  \put(35,1){\line(-2,3){12.7}}
  \put(45,1){\line(2,3){12,7}}
   \put(39,2){$v$}
  \put(40,-2){\oval(30,4)[b]}
  \put(39,-7){$u$}
   \put(40,22){\oval(55,4)[t]}
  \put(33,26){$d_{A(x)}(u)$}
   \put(25,1){\line(-2,3){12.7}}
  \put(55,1){\line(2,3){12,7}}
   \put(33,16){$d_{A(x)}(v)$}
\end{picture}
\end{center}
\caption{Proof of Proposition~\ref{infinite main}:
The word $v$ is extended to a lean factor $u$ of $d_{A(x)}(u)$.}\label{Figure 4}
\end{figure}
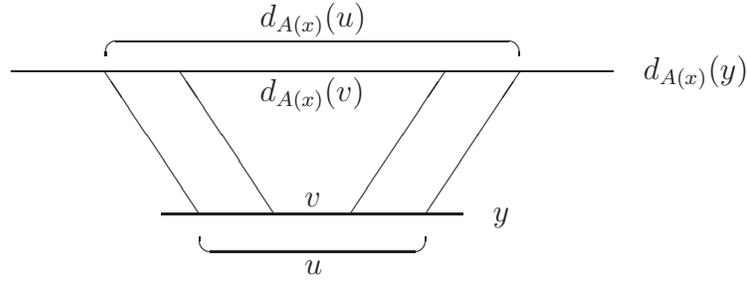

\begin{proof}[Proof of Corollary~\ref{infinite main}]
Let $x\in \{0,1\}^\nats$ be aperiodic. By Proposition~\ref{aperiodic2}, $\psi(x) \geq 3.$

If $x$ is double Sturmian then,  by Proposition~\ref{A}, $\mu(w) \le 3$ for all factors of $x$,
and thus by definition, $\psi(x)=3.$

For the converse, assume that $\psi(x) = 3$, and consider the set $A(x) \subseteq \{0,1\}$
together with the lean word $y$ of~$x.$ Then $x$ is a suffix of $d_{A(x)}(y)$. We need to
show that~$y$ is Sturmian. Let~$v$ be a factor of $y$, and let~$u$ be a factor of $y$ such
that $u$ contains $v$ and $d_{A(x)}(u)$ contains a factor of the form $ba^{2k+1}b$ with
$k\ge 0$ for each $a \notin A(x)$. Hence $A(d_{A(x)}(u))=A(x)$, and therefore $u$ is the
lean word of $d_{A(x)}(u)$; see Figure~\ref{Figure 4}. By Lemma~\ref{heritage},
$\mu(d_{A(x)}(u)) \le 3$, and thus Lemma~\ref{conlemma} yields that~$u$, and hence
also~$v$, is balanced. It follows that~$y$ is Sturmian as required.
\end{proof}

\goodbreak

\end{document}